\documentclass[11pt]{amsart}
\usepackage{amsmath}
\usepackage{amsfonts}
\usepackage[active]{srcltx}

\newtheorem{theorem}{Theorem}

\newtheorem{lemma}{Lemma}

\newtheorem{corollary}{Corollary}

\newtheorem*{Go}{Theorem G}

\begin{document}
\author{Tsitsino tepnadze}
\title[Approximation Properties of Ces\`{a}ro Means]{On the Approximation
Properties of Ces\`{a}ro Means of Negative Order of Vilenkin-Fourier Series}
\address{T.Tepnadze, Department of Mathematics, Faculty of Exact and Natural
Sciences, Ivane Javakhishvili Tbilisi State University, Chavcha\-vadze str.
1, Tbilisi 0128, Georgia}
\email{tsitsinotefnadze@gmail.com}
\maketitle

\begin{abstract}
In this paper we establish approximation properties of Ces\`{a}ro $%
(C,-\alpha )$ means with $\alpha $ $\epsilon $ $(0,1)$ of Vilenkin-Fourier
series.This result allows one to obtain a condition which is sufficient for
the convergence of the means $\sigma _{n}^{-\alpha }(f,x)$ to $f(x)$ in the $%
L^{p}-$metric.
\end{abstract}

\footnotetext{%
	2010 Mathematics Subject Classification 42C10 .
	}

Let $N_{+}$ denote the set of positive integers, $N:=N_{+}\cup \{0\}.$ Let $%
m:=\left( m_{0},m_{1},...\right) $ denote a sequence of positive integers
not lass then 2. Denote by $Z_{m_{k}}:=\{0,1,...,m_{k}-1\}$ the additive
group of integers modulo $m_{k}$. Define the group $G_{m}$ as the complete
direct product of the groups $Z_{m_{j}}$ with the product of the discrete
topologies of $Z_{mj}$'s.

The direct product of the measures

\begin{equation*}
\mu _{k}\left( \{j\}\right) :=\frac{1}{m_{k}}\text{ \ \ \ \ \ \ \ \ }\left( j%
\text{ }\in Z_{m_{k}}\right) .
\end{equation*}

Is the Haar measure on $G_{m}$ with $\mu \left( G_{m}\right) =1.$ If the
sequence $m$ is bounded, then $G_{m}$\ is called a bounded Vilenkin group.
In this paper we will consider only bounded Vilenkin group. The elements of $%
G_{m}$ can be represented by sequences $x:=\left(
x_{0},x_{1},...,x_{j},...\right) ,$ $\left( x_{j}\in Z_{m_{j}}\right) .$ The
group operation $+$ in $G_{m}$ is given by%
\begin{equation*}
x+y=\left( \left( x_{0}+y_{0\text{ }}\right) mod\text{ }m_{0},...,\left(
x_{k}+y_{k\text{ }}\right) mod\text{ }m_{k},...\right) ,
\end{equation*}%
where $x:=\left( x_{0},...,x_{k},...\right) $ and $y:=\left(
y_{0},...,y_{k},...\right) \in G_{m}.$

The inverse of $+$ will be denoted by $-.$It is easy to give a base for the
neighborhoods of $G_{m}:$

\begin{equation*}
I_{0}\left( x\right) :=G_{m}
\end{equation*}

\begin{equation*}
I_{n}\left( x\right) :=\{y\in G_{m}|y_{0\text{ }}=x_{0\text{ }},...,y_{n-1%
\text{ }}=x_{n-1\text{ }}\}.
\end{equation*}

for $x$ $\in $ $G_{m},$ $n$ $\in $ $N$ define $I_{n}=I_{n}\left( 0\right) $.
Set $e_{n}:=\left( 0,...,0,1,0,...\right) \in $ $G_{m}$ the $n$ th
coordinate of which is $1$ and the rest are zeros $\left( n\in N\right) .$

If we define the so-called generalized number system based on $m$ in the
following way: $M_{0}:=1,$ $M_{k+1}:=m_{k}M_{k}$ $\left( k\in N\right) ,$
then every $n$ $\in $ $N$ can be uniquely expressed as $n=\sum%
\limits_{j=0}^{\infty }n_{j}M_{j},$ where $n_{j}$ $\in \ Z_{mj}$ $\left(
j\in N_{+}\right) $ and only a finite number of $n_{j}$'s differ from zero.
We use the following notation. Let $\left\vert n\right\vert :=$max$\{k\in
N:n_{k}\neq 0\}$ (that is , $M_{|n|}\leq n<M_{|n|+1}$). Next, we introduce
of $G_{m}$ an orthonormal system which is called Vilenkin system. At first
define the complex valued functions $r_{k}\left( x\right) :G_{m}\rightarrow
C.$ the generalized Rademacher functions in this way 
\begin{equation*}
r_{k}(x):=\exp \left( \frac{2\pi ix_{k}}{m_{k}}\right) ,\text{\ \ \ \ \ \ }%
\left( i^{2}=-1,\text{ }x\in G_{m},\text{ }k\text{ }\in \text{ }N\right) .
\end{equation*}

Now define the Vilenkin system $\psi :=\left( \psi _{n}:n\in N\right) $ on $%
G_{m}$ as follows.

\begin{equation*}
\psi _{n}\left( x\right) :=\prod\limits_{k=0}^{\infty }r_{k}^{n_{k}}\left(
x\right) ,\text{\ \ \ \ \ \ \ }\left( n\text{ }\epsilon \text{ }N\right) .
\end{equation*}

In particular, we call the system the Walsh-Paley if $m=2.$ The Vilenkin
system is orthonormal and complete in \ $L^{1}\left( G_{m}\right) .$ Now,
introduce analogues of the usual definitions of the Fourier analysis. If $f$ 
$\in $ $L^{1}\left( G_{m}\right) $ we can establish the following
definitions in the usual way:

Fourier coefficients:

\begin{equation*}
\widehat{f}\left( k\right) :=\int\limits_{G_{m}}f\overline{\psi _{k}}d\mu ,\
\ \ \ \ \ \left( k\text{ }\in N\right) ,
\end{equation*}

partial sums:

\begin{equation*}
S_{n}f:=\sum\limits_{k=0}^{n-1}\widehat{f}\left( k\right) \psi _{k},\ \ \ \
\left( n\text{ }\in \text{ }N_{+}\text{ },\text{ }S_{0}f:=0\right) ,
\end{equation*}

Fej\'{e}r means:

\begin{equation*}
\sigma _{n}f:=\sum\limits_{k=0}^{n}\left( 1-\frac{k}{n+1}\right) \ \widehat{f%
}\left( k\right) \psi _{k},\ \ \ \ \left( n\text{ }\epsilon \text{ }%
N_{+}\right) .
\end{equation*}

Dirichlet kernels:

\begin{equation*}
D_{n}:=\sum\limits_{k=0}^{n-1}\psi _{k},\ \ \ \ \ \left( n\text{ }\in
N_{+}\right) .
\end{equation*}

Fej\'{e}r kernels:

\begin{equation*}
K_{n}(x):=\frac{1}{n}\sum\limits_{k=1}^{n}D_{k}\left( x\right) .
\end{equation*}

Recall that (see \cite{Gol} or \cite{Sw})

\begin{equation}
\quad \hspace*{0in}D_{M_{n}}\left( x\right) =\left\{ 
\begin{array}{l}
\text{ }M_{n},\text{\thinspace \thinspace \thinspace \thinspace if\thinspace
\thinspace }x\in I_{n}, \\ 
\text{ }0,\text{\thinspace \thinspace \thinspace \thinspace \thinspace if
\thinspace \thinspace }x\notin I_{n}.%
\end{array}%
\right.  \label{for1}
\end{equation}

It is well Known that

\begin{equation*}
\sigma _{n}f(x)=\int\limits_{G_{m}}f\left( t\right) K_{n}\left( x-t\right)
d\mu \left( t\right) .
\end{equation*}

The $(C,-\alpha )$ means of the Vilenkin-Fourier series are defined as

\begin{equation*}
\sigma _{n}^{-\alpha }\left( f,x\right) =\frac{1}{A_{n}^{-\alpha }}%
\sum\limits_{k=0}^{n}A_{n-k}^{-\alpha }\widehat{f}\left( k\right) \psi
_{k}\left( x\right) ,
\end{equation*}%
where

\begin{equation*}
A_{0}^{\alpha }=1,\ \ \ \ \ \ \ A_{n}^{\alpha }=\frac{\left( \alpha
+1\right) ...\left( \alpha +n\right) }{n!}.
\end{equation*}

It is well Known that \cite{Zy}

\begin{equation}
A_{n}^{\alpha }=\sum\limits_{k=0}^{n}A_{k}^{\alpha -1}.\text{ \ \ \ \ \ \ }
\label{for2}
\end{equation}

\begin{equation}
A_{n}^{\alpha }-A_{n-1}^{\alpha }=A_{n}^{\alpha -1}.\ \ \ \ \ \ \ 
\label{for3}
\end{equation}

\begin{equation}
A_{n}^{\alpha }\sim n^{\alpha }.\ \ \ \ \ \ \   \label{for4}
\end{equation}

The norm of the space $L^{p}\left( G_{m}\right) $ is defined by

\begin{equation*}
\left\Vert f\right\Vert _{p}:=\left( \int\limits_{G_{m}}\left\vert f\left(
x\right) \right\vert ^{p}d\mu \left( x\right) \right) ^{1/p},\ \ \ \left(
1\leq p<\infty \right) .
\end{equation*}

Denote by $C\left( G_{m}\right) $ the class of continuous functions on the
group $G_{m}$, endoved with the supremum norm.

For the sake of brevity in notation, we agree to write $L^{\infty }\left(
G_{m}\right) $ instead of $C\left( G_{m}\right) .$

\bigskip Let $f\in L^{p}\left( G_{m}\right) ,1\leq p\leq \infty .$ The
expression 
\begin{equation*}
\omega \left( \frac{1}{M_{n}},f\right) _{p}=\sup_{h\in I_{n}}\left\Vert
f\left( \cdot +h\right) -f\left( \cdot \right) \right\Vert _{p}
\end{equation*}%
is called the modulus of continuity.

The problems of summability of partial sums and Ces\`{a}ro \ means for
Walsh-Fourier series were studied in \cite{Fi}, \cite{GoJAT}-\cite{Su},\cite%
{Tev}. In his monography \cite{Zh} Zhizhinashvili investigated the behavior
of Ces\`{a}ro method of negative order for trigonometric Fourier series in
detail. Goginava \cite{GoJAT} studied analogical question in case fo the
Walsh system. In particular, the following theorem is proved.

\begin{Go}
\cite{GoJAT}Let $f$ belong to $L^{p}\left( G_{2}\right) $ for some $p$ $\in $
$\left[ 1,\infty \right] $ and \ $\alpha $ $\in $ $\left( 0,1\right) $. Then
for any \ $2^{k}\leq n<2^{k+1}$ $(k,n\in N)$\ the inequality
\end{Go}

\begin{equation*}
\left\Vert \sigma _{n}^{-\alpha }\left( f\right) -f\right\Vert _{p}\leq
c\left( p,\alpha \right) \left\{ 2^{k\alpha }\omega \left(
1/2^{k-1},f\right) _{p}+\sum\limits_{r=0}^{k-2}2^{r-k}\omega \left(
1/2^{r},f\right) _{p}\right\}
\end{equation*}

holds true.

In this paper we establish approximation properties of Ces\`{a}ro $%
(C,-\alpha )$ means with $\alpha $ $\epsilon $ $(0,1)$ of Vilenkin-Fourier
series.

\begin{theorem}
\label{T2}Let $f$ belong to $L^{p}\left( G_{m}\right) $ for some $p$ $\in $ $%
\left[ 1,\infty \right] $ and \ $\alpha $ $\in $ $\left( 0,1\right) $. Then
for any \ $M_{k}\leq n<M_{k+1}$ $(k,n\in N)$\ the inequality
\end{theorem}

\begin{equation*}
\left\Vert \sigma _{n}^{-\alpha }\left( f\right) -f\right\Vert _{p}\leq
c\left( p,\alpha \right) \left\{ M_{k}^{\alpha }\omega \left(
1/M_{k-1},f\right) _{p}+\sum\limits_{r=0}^{k-2}\frac{M_{r}}{M_{k}}\omega
\left( 1/M_{r},f\right) _{p}\right\}
\end{equation*}

holds true.

This result allows one to obtain the condition which is sufficient for the
convergence of the means $\sigma _{n}^{-\alpha }(f,x)$ to $f(x)$ in the $%
L^{p}-$metric.

\begin{corollary}
\label{C1}Let $f$\ belong to $L^{p}\left( G_{m}\right) $ for some $p$ $\in $ 
$\left[ 1,+\infty \right] $ and let
\end{corollary}

\begin{equation*}
M_{k}^{\alpha }\omega \left( 1/M_{k-1},f\right) _{p}\rightarrow 0,\ \ \ \ \
as\ \ \ \ \ k\rightarrow \infty ,\ \ \ \ \ \left( 0<\alpha <1\right) .
\end{equation*}

Then

\begin{equation*}
\left\Vert \sigma _{n}^{-\alpha }\left( f\right) -f\right\Vert
_{p}\rightarrow 0\ \ \ \ \ as\ \ \ \ \ n\rightarrow \infty .
\end{equation*}

In order to prove Theorem \ref{T2} we need the following lemmas.

\begin{lemma}
\label{L1}\cite{AVDR}Let $\alpha _{1},...,\alpha _{n}$ be real numbers.Then
\end{lemma}

\begin{equation*}
\frac{1}{n}\int\limits_{G}\left\vert \sum\limits_{k=1}^{n}\alpha
_{k}D_{k}(x)\right\vert d\mu (x)\leq \frac{c}{\sqrt{n}}\left(
\sum\limits_{k=1}^{n}\alpha _{k}^{2}\right) ^{1/2}.
\end{equation*}

where $c$ is an absolute constant.

\begin{lemma}
\label{L3}Let $f$ $\in L^{p}(G_{m})$ for some $p$ $\in $ $\left[ 1,\infty %
\right] .$ Then for every $\alpha $ $\in $ $\left( 0,1\right) $ the
following estimations holds%
\begin{equation*}
\frac{1}{A_{n}^{-\alpha }}\left\Vert
\int\limits_{G_{m}}\sum\limits_{v=0}^{M_{k-1}-1}A_{n-v}^{-\alpha }\psi
_{v}\left( u\right) \left[ f\left( \cdot +u\right) -f\left( \cdot \right) %
\right] d\mu \left( u\right) \right\Vert _{p}
\end{equation*}%
\begin{equation*}
\leq c\left( p,\alpha \right) \sum\limits_{r=0}^{k-1}\frac{M_{r}}{M_{k}}%
\omega \left( 1/M_{k},f\right) _{p},
\end{equation*}
\end{lemma}

where $M_{k}\leq n\leq M_{k+1}.$

\begin{proof}[Proof of Lemma \protect\ref{L3}]
Applying Abel's transformation, from (\ref{for3}) we get

\begin{equation}
\text{\ }\frac{1}{A_{n}^{-\alpha }}\left\Vert
\int\limits_{G_{m}}\sum\limits_{v=0}^{M_{k-1}-1}A_{n-v}^{-\alpha }\psi
_{v}\left( u\right) \left[ f\left( \cdot +u\right) -f\left( \cdot \right) %
\right] d\mu \left( u\right) \right\Vert _{p}\text{\ \ \ }  \label{for5}
\end{equation}

\begin{equation*}
=\frac{1}{A_{n}^{-\alpha }}\left\Vert
\int\limits_{G_{m}}\sum\limits_{v=1}^{M_{k-1}}A_{n-v-1}^{-\alpha }\psi
_{v-1}\left( u\right) \left[ f\left( \cdot +u\right) -f\left( \cdot \right) %
\right] d\mu \left( u\right) \right\Vert _{p}\text{ \ \ }
\end{equation*}

\begin{equation*}
\leq \frac{1}{A_{n}^{-\alpha }}\left\Vert
\int\limits_{G_{m}}\sum\limits_{v=1}^{M_{k-1}-1}A_{n-v-1}^{-\alpha
-1}D_{v}\left( u\right) \left[ f\left( \cdot +u\right) -f\left( \cdot
\right) \right] d\mu \left( u\right) \right\Vert _{p}
\end{equation*}

\begin{equation*}
+\frac{1}{A_{n}^{-\alpha }}\left\Vert
\int\limits_{G_{m}}A_{n-M_{k-1}-1}^{-\alpha }D_{M_{k-1}}\left( u\right) %
\left[ f\left( \cdot +u\right) -f\left( \cdot \right) \right] d\mu \left(
u\right) \right\Vert _{p}
\end{equation*}%
\begin{equation*}
=I_{1}+I_{2}.
\end{equation*}

From the generalized Minkowski's inequality, and by (\ref{for1}) and (\ref%
{for4}) we obtain

\begin{equation}
I_{2}\leq c(\alpha )M_{k-1}\int\limits_{I_{k-1}}\left\Vert \left[ f\left(
\cdot +u\right) -f\left( \cdot \right) \right] \right\Vert _{p}d\mu \left(
u\right)  \label{for6}
\end{equation}

\begin{equation*}
=O\left( \omega \left( 1/M_{k-1},f\right) \right) _{p},
\end{equation*}

\begin{equation}
\text{\ \ }I_{1}\leq \frac{1}{A_{n}^{-\alpha }}\sum\limits_{r=0}^{k-2}\left%
\Vert \int\limits_{G_{m}}\sum\limits_{v=M_{r}}^{M_{r+1}-1}A_{n-v-1}^{-\alpha
-1}D_{v}\left( u\right) \left[ f\left( \cdot +u\right) -f\left( \cdot
\right) \right] d\mu \left( u\right) \right\Vert _{p}\text{ }  \label{for7}
\end{equation}

\begin{equation*}
\leq \frac{1}{A_{n}^{-\alpha }}\sum\limits_{r=0}^{k-2}\left\Vert
\int\limits_{G_{m}}\sum\limits_{v=M_{r}}^{M_{r+1}-1}A_{n-v-1}^{-\alpha
-1}D_{v}\left( u\right) \left[ f\left( \cdot +u\right) -S_{M_{r}}\left(
\cdot +u,f\right) \right] d\mu \left( u\right) \right\Vert _{p}
\end{equation*}

\begin{equation*}
+\frac{1}{A_{n}^{-\alpha }}\sum\limits_{r=0}^{k-2}\left\Vert
\int\limits_{G_{m}}\sum\limits_{v=M_{r}}^{M_{r+1}-1}A_{n-v-1}^{-\alpha
-1}D_{v}\left( u\right) \left[ S_{M_{r}}\left( \cdot +u,f\right)
-S_{M_{r}}\left( \cdot ,f\right) \right] d\mu \left( u\right) \right\Vert
_{p}
\end{equation*}

\begin{equation*}
+\frac{1}{A_{n}^{-\alpha }}\sum\limits_{r=0}^{k-2}\left\Vert
\int\limits_{G_{m}}\sum\limits_{v=M_{r}}^{M_{r+1}-1}A_{n-v-1}^{-\alpha
-1}D_{v}\left( u\right) [S_{M_{r}}\left( \cdot ,f\right) -f\left( \cdot
\right) ]d\mu \left( u\right) \right\Vert _{p}
\end{equation*}

\begin{equation*}
=I_{11}+I_{12}+I_{13}.
\end{equation*}

Since 
\begin{equation*}
\left\Vert f-S_{M_{r}}\left( f\right) \right\Vert _{p}\leq \omega \left(
1/M_{r},f\right) _{p},
\end{equation*}

using Lemma \ref{L1} for $I_{11}$we can write

\begin{equation}
I_{11}  \label{for8}
\end{equation}

\begin{equation*}
\leq \frac{1}{A_{n}^{-\alpha }}\sum\limits_{r=0}^{k-2}\int\limits_{G_{m}}%
\left\vert \sum\limits_{v=M_{r}}^{M_{r+1}-1}A_{n-v-1}^{-\alpha
-1}D_{v}\left( u\right) \right\vert \left\Vert f\left( \cdot +u\right)
-S_{M_{r}}\left( \cdot +u,f\right) \right\Vert _{p}d\mu \left( u\right)
\end{equation*}

\begin{equation*}
\leq \frac{1}{A_{n}^{-\alpha }}\sum\limits_{r=0}^{k-2}\omega \left(
1/M_{r},f\right) _{p}\int\limits_{G_{m}}\left\vert
\sum\limits_{v=M_{r}}^{M_{r+1}-1}A_{n-v-1}^{-\alpha -1}D_{v}\left( u\right)
\right\vert d\mu \left( u\right)
\end{equation*}

\begin{equation*}
\leq c\left( \alpha \right) n^{\alpha }\sum\limits_{r=0}^{k-2}\omega \left(
1/M_{r},f\right) _{p}\sqrt{M_{r+1}}\left(
\sum\limits_{v=M_{r}}^{M_{r+1}-1}\left( n-v-1\right) ^{-2\alpha -2}\right)
^{1/2}
\end{equation*}

\begin{equation*}
\leq c\left( \alpha \right) n^{\alpha }\sum\limits_{r=0}^{k-2}\omega \left(
1/M_{r},f\right) _{p}\sqrt{M_{r+1}}\left( n-M_{r+1}\right) ^{-\alpha -1}%
\sqrt{M_{r+1}}
\end{equation*}

\begin{equation*}
\leq c\left( \alpha \right) \sum\limits_{r=0}^{k-2}\frac{M_{r}}{M_{k}}\omega
\left( 1/M_{r},f\right) _{p}.
\end{equation*}

Analogously, we can prove that

\begin{equation}
I_{13}\leq c\left( \alpha \right) \sum\limits_{r=0}^{k-2}\frac{M_{r}}{M_{k}}%
\omega \left( 1/M_{r},f\right) _{p}.  \label{for9}
\end{equation}

It is evident that

\begin{equation}
\text{\ }\int\limits_{G_{m}}\sum\limits_{v=M_{r}}^{M_{r+1}-1}A_{n-v-1}^{-%
\alpha -1}D_{v}\left( u\right) \left[ S_{M_{r}}\left( x+u,f\right)
-S_{M_{r}}\left( x,f\right) \right] d\mu \left( u\right)   \label{for9.1}
\end{equation}

\begin{equation*}
=\sum\limits_{v=M_{r}}^{M_{r+1}-1}A_{n-v-1}^{-\alpha
-1}\int\limits_{G_{m}}S_{M_{r}}\left( x+u,f\right) D_{v}\left( u\right) d\mu
\left( u\right) -\sum\limits_{v=M_{r}}^{M_{r+1}-1}A_{n-v-1}^{-\alpha
-1}S_{M_{r}}\left( x,f\right) 
\end{equation*}

\begin{equation*}
=\sum\limits_{v=M_{r}}^{M_{r+1}-1}A_{n-v-1}^{-\alpha -1}S_{v}\left(
x,S_{M_{r}}\left( f\right) \right)
-\sum\limits_{v=M_{r}}^{M_{r+1}-1}A_{n-v-1}^{-\alpha -1}S_{M_{r}}\left(
x,f\right) 
\end{equation*}

\begin{equation*}
=\sum\limits_{v=M_{r}}^{M_{r+1}-1}A_{n-v-1}^{-\alpha -1}S_{M_{r}}\left(
x,f\right) -\sum\limits_{v=M_{r}}^{M_{r+1}-1}A_{n-v-1}^{-\alpha
-1}S_{M_{r}}\left( x,f\right) =0.
\end{equation*}

Hence,$\bigskip $%
\begin{equation}
I_{12}=0.  \label{for9.2}
\end{equation}

Combining (\ref{for5})-(\ref{for9.2}) we receive the proof of Lemma \ref{L3}.
\end{proof}

\begin{lemma}
\label{L4} Let $f$ $\in L^{p}(G_{m})$ for some $p$ $\in $ $\left[ 1,\infty %
\right] .$ Then for every $\alpha $ $\in $ $\left( 0,1\right) $ the
following estimations hold
\end{lemma}

\begin{equation*}
\frac{1}{A_{n}^{-\alpha }}\left\Vert
\int\limits_{G_{m}}\sum\limits_{v=M_{k-1}}^{M_{k}-1}A_{n-v}^{-\alpha }\psi
_{v}\left( u\right) \left[ f\left( \cdot +u\right) -f\left( \cdot \right) %
\right] d\mu \left( u\right) \right\Vert _{p}
\end{equation*}

\begin{equation*}
\leq c\left( p,\alpha \right) \omega \left( 1/M_{k-1},f\right)
_{p}M_{k}^{\alpha }
\end{equation*}

and

\begin{equation*}
\frac{1}{A_{n}^{-\alpha }}\left\Vert
\int\limits_{G_{m}}\sum\limits_{v=M_{k}}^{n}A_{n-v}^{-\alpha }\psi
_{v}\left( u\right) \left[ f\left( \cdot +u\right) -f\left( \cdot \right) %
\right] d\mu \left( u\right) \right\Vert _{p}
\end{equation*}

\begin{equation*}
\leq c\left( p,\alpha \right) \omega \left( 1/M_{k},f\right)
_{p}M_{k}^{\alpha }
\end{equation*}

where $M_{k}\leq n<M_{k+1}.$

\begin{proof}[Proof of Lemma \protect\ref{L4}]
We can write

\begin{equation}
II=\frac{1}{A_{n}^{-\alpha }}\left\Vert
\int\limits_{G_{m}}\sum\limits_{v=M_{k-1}}^{M_{k}-1}A_{n-v}^{-\alpha }\psi
_{v}\left( u\right) \left[ f\left( \cdot +u\right) -f\left( \cdot \right) %
\right] d\mu \left( u\right) \right\Vert _{p}  \label{for11}
\end{equation}

\begin{equation*}
=\frac{1}{A_{n}^{-\alpha }}\left\Vert
\int\limits_{G_{m}}\sum\limits_{v=M_{k-1}}^{M_{k}-1}A_{n-v}^{-\alpha }\psi
_{v}\left( u\right) f\left( \cdot +u\right) d\mu \left( u\right) \right\Vert
_{p}
\end{equation*}%
\begin{equation*}
\leq \frac{1}{A_{n}^{-\alpha }}\left\Vert
\int\limits_{G_{m}}\sum\limits_{v=M_{k-1}}^{M_{k}-1}A_{n-v}^{-\alpha }\psi
_{v}\left( u\right) \left[ f\left( \cdot +u\right) -S_{M_{k-1}}\left( \cdot
+u,f\right) \right] d\mu \left( u\right) \right\Vert _{p}
\end{equation*}

\begin{equation*}
+\frac{1}{A_{n}^{-\alpha }}\left\Vert
\int\limits_{G_{m}}\sum\limits_{v=M_{k-1}}^{M_{k}-1}A_{n-v}^{-\alpha }\psi
_{v}\left( u\right) S_{M_{k-1}}\left( \cdot +u,f\right) d\mu \left( u\right)
\right\Vert _{p}=II_{1}+II_{2}.
\end{equation*}

Since%
\begin{equation}
\int\limits_{G_{m}}\sum\limits_{v=M_{k-1}}^{M_{k}-1}A_{n-v}^{-\alpha }\psi
_{v}\left( u\right) S_{M_{k-1}}\left( x+u,f\right) d\mu \left( u\right) 
\label{for11.1}
\end{equation}

\begin{equation*}
=\sum\limits_{j=0}^{M_{k-1}-1}\widehat{f}\left( j\right) \psi _{j}\left(
x\right) \sum\limits_{v=M_{k-1}}^{M_{k}-1}A_{n-v}^{-\alpha
}\int\limits_{G_{m}}\psi _{v}\left( u\right) \psi _{j}\left( u\right) d\mu
\left( u\right) =0,
\end{equation*}%
for $II_{2}$ we obtain

\begin{equation}
II_{2}=0.  \label{for12}
\end{equation}

Using generalized Minkowski's inequality we have

\begin{eqnarray}
\ \text{\ }II_{1} &\leq &\frac{1}{A_{n}^{-\alpha }}\int\limits_{G_{m}}\left%
\vert \sum\limits_{v=M_{k-1}}^{M_{k}-1}A_{n-v}^{-\alpha }\psi _{v}\left(
u\right) \right\vert  \label{for13} \\
&&\times \left\Vert \left[ f\left( \cdot \oplus u\right) -S_{M_{k-1}}\left(
\cdot \oplus u,f\right) \right] \right\Vert _{p}d\mu \left( u\right)  \notag
\end{eqnarray}

\begin{equation*}
\leq c\left( \alpha \right) n^{\alpha }\omega \left( 1/M_{k-1},f\right)
_{p}\int\limits_{G_{m}}\left\vert
\sum\limits_{v=M_{k-1}}^{M_{k}-1}A_{n-v}^{-\alpha }\psi _{v}\left( u\right)
\right\vert d\mu \left( u\right) .
\end{equation*}

Let $t\in I_{A-1}\backslash I_{A},\,\ A=1,2,...,k-1\ $and $\ M_{k}=pM_{A}+q,$
where $0\leq q<M_{A}.$Then we have%
\begin{eqnarray}
&&\sum\limits_{v=M_{k-1}}^{M_{k}-1}A_{n-v}^{-\alpha }\psi _{v}\left( t\right)
\label{for14} \\
&=&\sum\limits_{v=M_{k-1}}^{pM_{A}-1}A_{n-v}^{-\alpha }\psi _{v}\left(
t\right) +\sum\limits_{v=pM_{A}}^{M_{k}-1}A_{n-v}^{-\alpha }\psi _{v}\left(
t\right)  \notag \\
&=&\sum\limits_{r=M_{k-1}/M_{A}}^{p-1}\sum\limits_{v=rM_{A}}^{\left(
r+1\right) M_{A}-1}A_{n-v}^{-\alpha }\psi _{v}\left( t\right)  \notag \\
&&+\sum\limits_{v=0}^{q-1}A_{n-v-pM_{A}}^{-\alpha }\psi _{v+pM_{A}}\left(
t\right)  \notag \\
&=&\sum\limits_{r=M_{k-1}/M_{A}}^{p-1}\sum%
\limits_{v=0}^{M_{A}-1}A_{n-v-rM_{A}}^{-\alpha }\psi _{v+rM_{A}}\left(
t\right)  \notag \\
&&+\sum\limits_{v=0}^{q-1}A_{n-v-pM_{A}}^{-\alpha }\psi _{v+pM_{A}}\left(
t\right)  \notag
\end{eqnarray}

\begin{eqnarray}
&=&\sum\limits_{r=M_{k-1}/M_{A}}^{p-2}\psi _{rM_{A}}\left( t\right)
\sum\limits_{v=0}^{M_{A}-1}A_{n-v-rM_{A}}^{-\alpha }\psi _{v}\left( t\right)
\notag \\
&&+\psi _{\left( p-1\right) M_{A}}\left( t\right)
\sum\limits_{v=0}^{M_{A}-1}A_{n-v-\left( p-1\right) M_{A}}^{-\alpha }\psi
_{v}\left( t\right)  \notag \\
&&+\psi _{rM_{A}}\left( t\right) \sum\limits_{v=0}^{q-1}A_{q-v}^{-\alpha
}\psi _{v}\left( t\right)  \notag \\
&=&A_{1}+A_{2}+A_{3}.  \notag
\end{eqnarray}

Since $D_{M_{A}}\left( t\right) =0,t\in I_{A-1}\backslash I_{A}$ and $%
\left\vert D_{k}\left( t\right) \right\vert \leq k,$ from Abel's
transformation,we have

\begin{eqnarray}
\left\vert A_{1}\right\vert &=&\left\vert
\sum\limits_{r=M_{k-1}/M_{A}}^{p-2}\psi _{rM_{A}}\left( t\right)
\sum\limits_{v=0}^{M_{A}-2}A_{n-v-rM_{A}}^{-\alpha -1}D_{v}\left( t\right)
\right\vert  \label{for15} \\
&\leq &c\left( \alpha \right)
M_{A}\sum\limits_{r=M_{k-1}/M_{A}}^{p-2}\sum\limits_{v=0}^{M_{A}}\left(
n-rM_{A}-v\right) ^{-\alpha -1}  \notag \\
&\leq &c\left( \alpha \right) M_{A}\left( n-\left( p-1\right) M_{A}\right)
^{-\alpha }  \notag \\
&\leq &c\left( \alpha \right) M_{A}^{1-\alpha }  \notag
\end{eqnarray}

For $A_{2}$ we have

\begin{eqnarray}
\left\vert A_{2}\right\vert &\leq &c\left( \alpha \right)
\sum\limits_{v=0}^{M_{A}-1}\left( n-\left( p-1\right) M_{A}-v\right)
^{-\alpha }  \label{for16} \\
&\leq &c\left( \alpha \right) \sum\limits_{v=0}^{M_{A}-1}\left(
M_{A}+q-v\right) ^{-\alpha }\leq c\left( \alpha \right) M_{A}^{1-\alpha }. 
\notag
\end{eqnarray}

We can write

\begin{eqnarray}
\left\vert A_{3}\right\vert &\leq &c\left( \alpha \right)
\sum\limits_{v=0}^{q-1}\left( q-v\right) ^{-\alpha }  \label{for17} \\
&\leq &c\left( \alpha \right) q^{1-\alpha }\leq c\left( \alpha \right)
M_{A}^{1-\alpha }.  \notag
\end{eqnarray}

Combining (\ref{for14})-(\ref{for17}) \ we obtain

\begin{equation}
\left\vert \sum\limits_{v=M_{k-1}}^{M_{k}-1}A_{n-v}^{-\alpha }\psi
_{v}\left( u\right) \right\vert \leq c\left( \alpha \right) M_{A}^{1-\alpha
},\text{ }t\in I_{A-1}\backslash I_{A},A=1,...,k-1.\text{\ }  \label{for18}
\end{equation}

Hence, we can write

\begin{eqnarray}
&&\int\limits_{G_{m}}\left\vert
\sum\limits_{v=M_{k-1}}^{M_{k}-1}A_{n-v}^{-\alpha }\psi _{v}\left( u\right)
\right\vert d\mu \left( u\right)  \label{for19} \\
&=&\sum\limits_{A=1}^{k-1}\int\limits_{I_{A-1}\backslash I_{A}}\left\vert
\sum\limits_{v=M_{k-1}}^{M_{k}-1}A_{n-v}^{-\alpha }\psi _{v}\left( u\right)
\right\vert d\mu \left( u\right)  \notag \\
&&+\int\limits_{I_{_{k-1}}}\left\vert
\sum\limits_{v=M_{k-1}}^{M_{k}-1}A_{n-v}^{-\alpha }\psi _{v}\left( u\right)
\right\vert d\mu \left( u\right)  \notag \\
&\leq &c\left( \alpha \right) \sum\limits_{A=1}^{k}\frac{1}{M_{A}}%
M_{A}^{1-\alpha }+\frac{c\left( \alpha \right) }{M_{k}}M_{k}^{1-\alpha } 
\notag \\
&\leq &c\left( \alpha \right) .  \notag
\end{eqnarray}

\bigskip Combining (\ref{for18})-(\ref{for19}) we have

\begin{equation}
\ II_{1}\leq c\left( \alpha \right) \omega \left( 1/M_{k-1},f\right)
_{p}M_{k}^{\alpha }.  \label{for20}
\end{equation}

Combining (\ref{for13}), (\ref{for11}) and (\ref{for20}) we conclude that

\begin{equation*}
\frac{1}{A_{n}^{-\alpha }}\left\Vert
\int\limits_{G_{m}}\sum\limits_{v=M_{k-1}}^{M_{k}-1}A_{n-v}^{-\alpha }\psi
_{v}\left( u\right) \left[ f\left( \cdot +u\right) -f\left( \cdot \right) %
\right] d\mu \left( u\right) \right\Vert _{p}
\end{equation*}

\begin{equation*}
\leq c\left( \alpha \right) \omega \left( 1/M_{k-1},f\right)
_{p}M_{k}^{\alpha }.
\end{equation*}

Analogously, we can prove that

\begin{equation*}
\frac{1}{A_{n}^{-\alpha }}\left\Vert
\int\limits_{G_{m}}\sum\limits_{v=M_{k}}^{n}A_{n-v}^{-\alpha }\psi
_{v}\left( u\right) \left[ f\left( \cdot +u\right) -f\left( \cdot \right) %
\right] d\mu \left( u\right) \right\Vert _{p}
\end{equation*}

\begin{equation*}
\leq c\left( \alpha \right) \omega \left( 1/M_{k},f\right) _{p}M_{k}^{\alpha
}.
\end{equation*}

\bigskip\ Lemma \ref{L4} proved.
\end{proof}

\begin{proof}[Proof of Theorem \protect\ref{T2}]
We can write

\begin{eqnarray*}
&&\sigma _{n}^{-\alpha }(f,x)-f(x) \\
&=&\frac{1}{A_{n}^{-\alpha }}\int\limits_{G_{m}}\sum%
\limits_{v=0}^{M_{k-1}-1}A_{n-v}^{-\alpha }\psi _{v}\left( x\right) \left[
f\left( \cdot +u\right) -f\left( \cdot \right) \right] d\mu \left( u\right)
\\
&&+\frac{1}{A_{n}^{-\alpha }}\int\limits_{G_{m}}\sum%
\limits_{v=M_{k-1}}^{M_{k}-1}A_{n-v}^{-\alpha }\psi _{v}\left( x\right) %
\left[ f\left( \cdot +u\right) -f\left( \cdot \right) \right] d\mu \left(
u\right) \\
&&+\frac{1}{A_{n}^{-\alpha }}\int\limits_{G_{m}}\sum%
\limits_{v=M_{k}}^{n}A_{n-v}^{-\alpha }\psi _{v}\left( x\right) \left[
f\left( \cdot +u\right) -f\left( \cdot \right) \right] d\mu \left( u\right)
\\
&=&I+II+III
\end{eqnarray*}

Since

\begin{equation*}
\left\Vert \sigma _{n}^{-\alpha }\left( f,\cdot \right) -f\left( \cdot
\right) \right\Vert _{p}\leq \left\Vert I\right\Vert _{p}+\left\Vert
II\right\Vert _{p}+\left\Vert III\right\Vert _{p}
\end{equation*}

From Lemmas \ref{L3} and \ref{L4} the proof of theorem is complete.
\end{proof}

\end{document}